\documentclass[10pt,a4paper]{article}
\usepackage[utf8]{inputenc}
\usepackage[english]{babel}
\usepackage{amsmath}
\usepackage{amsfonts}
\usepackage{amssymb}
\usepackage{amsthm}
\usepackage{graphicx}

\newtheorem{theorem}{Theorem}[section]
\newtheorem{thm}[theorem]{Theorem}

\newtheorem{definition}[theorem]{Definition} 
\newtheorem{lem}[theorem]{Lemma}

\theoremstyle{remark}

\numberwithin{claim}{theorem}

\begin{document}

 \title{On surfaces endowed with canonical principal direction in Euclidean spaces}

\author{Alev Kelleci\footnote{F\i rat University, Faculty of Science, Department of Mathematics, 23200 Elaz\i\u g, Turkey. e-mail: alevkelleci@hotmail.com},  Nurettin Cenk Turgay \footnote{Istanbul Technical University, Faculty of Science and Letters, Department of  Mathematics, 34469 Maslak, Istanbul, Turkey. e-mail: turgayn@itu.edu.tr} and Mahmut Erg\" ut\footnote{Nam\i k Kemal University, Faculty of Science and Letters, Department of Mathematics, 59030 Tekirda\u g, Turkey. email: mergut@nku.edu.tr}}

\maketitle

\begin{abstract}
In this paper, we introduce canonical principal direction 
(CPD) submanifolds with higher codimension in Euclidean spaces. 
We obtain the complete classification of surfaces endowed with CPD in the Euclidean 4-space.  


\noindent \textbf{Keywords.}
\end{abstract}

\section{Introduction}\label{sec:1}
Let $N$ be a Riemannian manifold, $M$ an immersed hypersurface of $N$ and $X$ a vector field in $N$. $M$  is said to have canonical principle direction (CPD) with relative to $X$ if the projection of $X$ onto the tangent space of $M$ gives one of principle directions of $M$, \cite{GarnicaPalnas}.  One of the most common examples of hypersurfaces with CPD is
rotational hypersurfaces in Euclidean spaces which have canonical principal direction relative $X$ if $X$ is chosen to be  a vector field parallel to its rotation axis. 

On the other hand, a submanifold in the Euclidean space is said to be a constant angle surface if there is a constant direction $k$ which makes constant angle with the tangent plane at every point of that surface. There are many classification results for such hypersurfaces called as constant angle (CA) hypersurfaces obtained so far, in different ambient spaces, \cite{CS2007,DillenMunt2009,FuNistor2013,GSK2011,AM2017,LopezMunt2011,MunteanuNistor2009}. Before we proceed, we would like to note that a CAS surface in the Euclidean 3-space has CPD relative to $k$. Because of this reason, hypersurfaces with CPD relative to a fixed direction in Euclidean spaces have caught interest of some geometers in the recent years. For example, surfaces with CPD in the  Euclidean 3-space $\mathbb E^3$ have been studied in \cite{MN2011}. Then, this study was moved into the Minkowski 3-space $\mathbb E^3_1$ in \cite{ANM2017,Nistor2013}. Furthermore, CPD surfaces in product spaces also take attention of some geometers. For example, some classification results on   surfaces with CPD relative to $\partial_t$ in $\mathbb S^2\times \mathbb R$ and $\mathbb H^2\times \mathbb R$ have been obtained in \cite{DFJvK2009,DMuntNistr2011} (See also \cite{FuNistor2013}), where $\partial_t$  denotes   the unit vector field tangent to the second factor. 

On the other hand, Tojeiro studied CPD hypersurfaces of $\mathbb S^n\times \mathbb R$ and $\mathbb H^n\times \mathbb R$ in \cite{To}. Later, Mendon\c ca and Tojeiro give generalization of the notion of CPD hypersurfaces into higher codimensional submanifolds.  For this porpose, they give the definition  class $\mathcal A$  in \cite{MT}. An immersion
 $f:M^n\rightarrow Q^n_c \times \mathbb R$ is said to belongs to class $\mathcal A$ 
immersions   if the tangential part of $\partial_t$ is one of principal directions
 of all shape operators of $f$. By a similar way, we would like to give the following definition of CPD submanifolds in Euclidean spaces.
\begin{definition}\label{DefoFCPD}
Let $M^n$ be a submanifold in $\mathbb{E}^{m}$ and $k$ be a fixed direction in $\mathbb{E}^{m}$.
 $M$ is said to be a submanifold endowed with canonical principal direction, (shortly, a CPD 
submanifold) if the tangential component $k^T$ of $k$ is one of principal directions of all shape 
operators of $M$.
\end{definition}

 The aim of this paper is to obtain complete classification of CPD surfaces  in the Euclidean 4-space $\mathbb E^4$.   In Sect. 2, we introduce the notation 
that we will use and give a brief summary of basic definitions in theory of 
submanifolds in Euclidean spaces. In Sect. 3, we obtain the complete classification
 of CPD surfaces   in the Euclidean 4-space.

\section{Prelimineries}
Let $\mathbb E^m$ denote the Euclidean $m$-space with the canonical Euclidean metric tensor given by  
$$
\widetilde g=\langle\ ,\ \rangle=\sum\limits_{i=1}^m dx_i^2,
$$
where $(x_1, x_2, \hdots, x_m)$  is a rectangular coordinate system in $\mathbb E^m$.  

Consider an $n$-dimensional Riemannian submanifold  of the space $\mathbb E^m$. We denote Levi-Civita connections of $\mathbb E^m$ and $M$ by $\widetilde{\nabla}$ and $\nabla$, respectively. The Gauss and Weingarten formulas are given, respectively, by
\begin{eqnarray}
\label{MEtomGauss} \widetilde\nabla_X Y&=& \nabla_X Y + h(X,Y),\\
\label{MEtomWeingarten} \widetilde\nabla_X \xi&=& -S_{\xi}(X)+D_X {\xi},  
\end{eqnarray}
whenever $X,Y$ are tangent and $\xi$ is normal vector field on $M$, where $h$,  $D$  and  $S$ are the second fundamental form, the normal connection and  the shape operator of $M$, respectively. It is well-known that the shape operator and the second fundamental form are related by  
$$\left\langle h(X, Y), {\xi} \right\rangle = \left\langle S_{\xi}X, Y \right\rangle.$$

The Gauss and Codazzi equations are given, respectively, by
\begin{eqnarray}
\label{MinkGaussEquation} \langle R(X,Y)Z,W\rangle&=&\langle h(Y,Z),h(X,W)\rangle-
\langle h(X,Z),h(Y,W)\rangle,\\
\label{MinkRicci} \langle R^D(X,Y)\xi,\eta \rangle&=&\langle \lbrack S_\xi,S_\eta]X,Y\rangle, \\
\label{MinkCodazzi} (\nabla_X h )(Y,Z)&=&(\nabla_Y h )(X,Z),
\end{eqnarray}
whenever $X,Y,Z,W$ are tangent to $M$, where $R,\; R^D$ are the curvature tensors associated with connections $\nabla$ 
and $D$, respectively. We note that  $\bar \nabla h$ is defined by
$$(\nabla_X h)(Y,Z)=D_X h(Y,Z)-h(\nabla_X Y,Z)-h(Y,\nabla_X Z).$$ 
A submanifold  $M$ is said to have flat normal bundle if $R^D=0$  identically.

The mean curvature vector field $H$ of the surface $M$ is defined as 
\begin{equation} \label{MC}
H=\frac{1}{2}tr h.
\end{equation}\label{GC}
If $M$ is a surface, i.e, $n=2$, then the Gaussian curvature $K$ of the surface $M^2$ is defined as 
\begin{equation}
K=\frac{R(X,Y,X,Y)}{Q(X,Y)},
\end{equation}
if $X$ and $Y$ are chosen so that $Q(X,Y)=\langle X,X\rangle\langle Y,Y\rangle-\langle X,Y\rangle^2$ does not vanish.

\section{CPD Surfaces in $\mathbb E^4$} \label{S:Classification}

In this section, we obtain classification of CPD surfaces in $\mathbb E^4$.

Let $M$ be a surface in $\mathbb E^4$ with CPD relative to $k$. Without loss of generality, we  assume that $k=(1,0,0,0)$. Then, one can define a tangent vector field $e_1$ and a normal vector field $e_3$ with the equation
\begin{equation} \label{CPDExpofS100}
k=\cos \theta e_1+\sin \theta {e_3}
\end{equation}
for a smooth function $\theta$. Let $e_2$ and ${e_4}$ be a unit tangent vector field and a unit normal vector field, satisfying $\langle e_1,e_2\rangle=0$ and $\left\langle {e_3},{e_4}\right\rangle=0$, respectively. By a simple computation considering \eqref{CPDExpofS100} we obtain the following lemma. Note that we put $h^\beta_{ij}=\langle h(e_i,e_j),e_\beta \rangle=\langle S_\beta e_i,e_j \rangle$, where $S_\beta=S_{e_\beta}$.
\begin{lem}\label{CPDCase1ClassThmDiagonSpacelikekClm1}
The Levi-Civita connection $\nabla$ of $M$ is given by
\begin{subequations} \label{CPDCASEISpacelikeLeviCivitaEq1ALL}
\begin{eqnarray}
\label{CPDCASEISpacelikeLeviCivitaEq1a}\nabla _{e_{1}}e_{1}=\nabla _{e_{1}}e_{2}=0, &&
\\\label{CPDCASEILeviCivitaEq1b}
\nabla _{e_{2}}e_{1}=\tan \theta h_{22}^3e_2, &\quad &\nabla _{e_{2}}e_{2}=-\tan \theta h_{22}^3e_1. 
\end{eqnarray}
\end{subequations}
and the matrix representations of shape operator $S$ of $M$ with respect to $\{e_1,e_2\}$ is
\begin{align}\label{CPDCASEISpacelikeLeviCivitaEq1}
\begin{split}
S_3=\left(\begin{array}{cc}
-e_1(\theta)&0\\
0&h_{22}^3
\end{array}\right), \quad
&S_4=\left(\begin{array}{cc}
0&0\\
0&h_{22}^4
\end{array}\right)
\end{split}
\end{align}
for functions $h_{11}^4$, $h_{12}^4$, $h_{22}^3$ and $h_{22}^4$ satisfying
\begin{subequations} \label{CPDClassThmDiagonSpacelikekCod1Case1ALL}
\begin{eqnarray} 
\label{CPDClassThmDiagonSpacelikekCod1Case11a} e_1(h_{22}^3)=\tan \theta h_{22}^3(h_{11}^3-h_{22}^3), \\
\label{CPDClassThmDiagonSpacelikekCod1Case11b} e_1(h_{22}^4)=-\tan \theta h_{22}^3 h_{22}^4, \\
\label{CPDClassThmDiagonSpacelikekCod1Case11c} h_{11}^4=0, \quad h_{12}^4=0.
\end{eqnarray}
\end{subequations}
Furthermore, $\theta$ satisfies
\begin{equation}\label{CPDCASEISpacelikeLeviCivitaEq1Theta}
e_2(\theta)=0.
\end{equation}
\end{lem}

\begin{proof}
By considering \eqref{CPDExpofS100} and the normal vector field $e_3$ being parallel, one can get
\begin{equation}\label{CPDApplyXExpofS100}
0=X(\cos \theta)e_1+\cos \theta \nabla_{X}e_1+%
\cos \theta h(e_1,X)-\sin \theta S_{3}X+X(\sin \theta){e_3}
\end{equation}
whenever $X$ is tangent to $M$. \eqref{CPDApplyXExpofS100} for $X=e_1$ gives
\begin{eqnarray}
\nonumber \nabla_{e_1}e_1=0,\quad \nabla_{e_1}e_2=0,\\
\label{CPDApplyXCPDExpofS1000Eq2b} h_{11}^3= -e_1(\theta),\\
\nonumber h_{11}^4=0.
\end{eqnarray}
while \eqref{CPDApplyXExpofS100} for $X=e_2$ is giving
\begin{eqnarray}
\nonumber\nabla_{e_2}e_1= \tan \theta h_{22}^3e_2, \quad \nabla_{e_2}e_2= -\tan \theta h_{22}^3e_1,\\
\nonumber h_{12}^4=0, \quad e_2(\theta)=0.
\end{eqnarray}
where $e_2$ is the other principal direction of $M$ corresponding with the principal curvature $h_{22}^3$. Thus, we have \eqref{CPDCASEISpacelikeLeviCivitaEq1ALL} and \eqref{CPDClassThmDiagonSpacelikekCod1Case11c} and \eqref{CPDCASEISpacelikeLeviCivitaEq1Theta} and the second fundamental form of $M$ becomes 
$$h(e_1,e_1)=-e_1(\theta)e_3,\quad h(e_1,e_2)=0,\quad\quad h(e_2,e_2)=h_{22}^3e_3+h_{22}^4e_4.$$
By considering the Codazzi equation \eqref{MinkCodazzi}, we obtain \eqref{CPDClassThmDiagonSpacelikekCod1Case11a} and \eqref{CPDClassThmDiagonSpacelikekCod1Case11b}.
\end{proof}
 
 Because of \eqref{CPDApplyXCPDExpofS1000Eq2b}, if $e_1(\theta)\equiv0$ implies $h_{11}^3=0$. We will consider this particular case seperately. 
 
First assume that $e_1(\theta)\neq0$. Let $p$  be a a point in $M$ at which $e_1(\theta)$ does not vanish. First, we would like to prove the following lemma.

\begin{lem}\label{CPDCase1ClassThmDiagonSpacelikekClm12}
There exists a local coordinate system  $(s,t)$ defined in a neighborhood $\mathcal N_p$ of $p$ such that the induced metric of $M$ is
\begin{equation}\label{CPDClassThmDiagonSpacelikekDefgEqRESCase1}
g=ds^2+m^2dt^2
\end{equation}
for a smooth function $m$ satisfying
\begin{equation}\label{CPDCase1ClassThmDiagonSpacelikekDefm}
e_1(m)-\tan \theta h_{22}^3m=0.
\end{equation}
Furthermore, the vector fields $e_1,e_2$ described above become $e_1=\partial_s$,  $\displaystyle e_2=\frac 1{m}\partial_t$ in $\mathcal N_p$.
\end{lem}

\begin{proof}
We have $[e_1,e_2]=-\tan \theta h_{22}^3e_2$ because of \eqref{CPDCASEISpacelikeLeviCivitaEq1ALL}. Thus, if $m$ is a non-vanishing smooth function on $M$ satisfying  \eqref{CPDCase1ClassThmDiagonSpacelikekDefm}, then we have $\displaystyle \left[e_1,me_2\right]=0$. Therefore, there exists a local coordinate system $(s,t)$ such that $e_1=\partial_s$ and  $\displaystyle e_2=\frac 1m\partial_t$. Thus, the induced metric of  $M$ is as given in \eqref{CPDClassThmDiagonSpacelikekDefgEqRESCase1}.
\end{proof}

Now, we are ready to obtain the classification theorem.
\begin{thm}\label{CPDClassThmDiagonSpacelikek}
Let $M$ be a regular surface in $\mathbb E^4$. Let  $M$ be a surface endowed with a canonical principal direction relative to $k=(1,0,0,0)$ and assume that the function $\theta$ defined in \eqref{CPDExpofS100} is not constant. Then, $M$ is congruent to the surface given by one of the followings
\begin{enumerate}
\item A surface given by
\begin{subequations}\label{CPDClassThmDiagonSpacelikekSurfaALL}
\begin{equation}
\label{CPDClassThmDiagonSpacelikekSurfaEq1}
x(s,t)=\Big(\int^{s}_{s_0}{\cos \theta(\tau)d\tau},{\phi_j}(t) \int^{s}_{s_0}{\sin \theta(\tau)d\tau}\Big)+\gamma(t), \quad j=2,3,4
\end{equation}
where $\gamma$ is the $\mathbb E^4$-valued function given by 
\begin{equation}\label{CPDClassThmDiagonSpacelikeSurfaEq2}
\gamma(t)=\Big(0,\int^t_{t_0}{\Psi(\tau){\phi_j}^{\prime}(\tau)d\tau}\Big).
\end{equation}
for a function $ \Psi\in C^{\infty}(M)$ and $\phi=\phi(t)$ is the unit speed curve lying on $\mathbb S^3(1)$ in $\mathbb E^4$;
\end{subequations}

\item A flat surface given by
\begin{align}\label{CPDClassThmDiagonSpacelikekSurfbEq1}
\begin{split}
x(s,t)=&\Big(\int^{s}_{s_0}{\cos \theta(\tau)d\tau},\phi_j(t_0)\int^{s}_{s_0}{\sin \theta(\tau)d\tau}\Big) +{t_0}\phi(t).
\end{split}
\end{align}
Here $\phi(t_0)$ and $\phi(t)$ are a constant vector and the unit speed curve lying on $\mathbb S^3(1)$ in $\mathbb E^4$, respectively.
\end{enumerate}

Conversely, surfaces described above are CPD relative to $k=(1,0,0,0).$
\end{thm}

\begin{proof}
In order to proof the necessary condition, we assume that $M$ is a surface endowed with a CPD relative to $k=(1,0,0,0)$ with the isometric immersion $x:M\rightarrow \mathbb E^4$. Let $\{e_1,e_2;e_3,e_4\}$ be the local orthonormal frame field described as before in Lemma \ref{CPDCase1ClassThmDiagonSpacelikekClm1}, $h_{11}^3,h_{22}^3$ and $h_{22}^4$ be the principal curvatures of $M$ and $(s,t)$ a local coordinate system given in Lemma \ref{CPDCase1ClassThmDiagonSpacelikekClm12}.

Note that, \eqref{CPDClassThmDiagonSpacelikekCod1Case11a}, \eqref{CPDClassThmDiagonSpacelikekCod1Case11b} and \eqref{CPDCase1ClassThmDiagonSpacelikekDefm} become, respectively
\begin{eqnarray}
\label{CPDCase1SpacelikeCods}(h_{22}^3)_s=-\tan \theta h_{22}^3(\theta' +h_{22}^3), \\
\label{CPDCase1SpacelikeCods2} (h_{22}^4)_s=-\tan \theta h_{22}^3 h_{22}^4, \\
\label{CPDCase1Spacelikems} m_s-m\tan \theta h_{22}^3=0, 
\end{eqnarray}
Moreover, we have 
\begin{equation}\label{CPDAftClm1Eq1SCase1}
e_1= x_s.
\end{equation}
By combining \eqref{CPDCase1Spacelikems} and \eqref{CPDCase1SpacelikeCods2} with \eqref{CPDCASEISpacelikeLeviCivitaEq1} we obtain
the shape operator $S$ of $M$  as
\begin{align} \label{CPDCASEISpacelikeShapeOpm1}
\begin{split}
S_3=\left( 
\begin{array}{cc}
-\theta' & 0 \\ 
0 & \cot \theta\frac{m_s}{m}%
\end{array}%
\right), \quad
&S_4=\left(\begin{array}{cc}
0&0\\
0&\frac{1}{m}
\end{array}\right)
\end{split}
\end{align}
where $'$ denotes ordinary differentiation with respect to the appropriated variable.

By combining \eqref{CPDCase1Spacelikems} and \eqref{CPDCase1SpacelikeCods} we obtain
$$m_{ss}-\theta'\cot \theta  m_s=0
$$
whose general solution is 
$$m(s,t)=\Psi_1(t)\int^{s}_{s_0}{\sin \theta(\tau)d\tau}+\Psi_2(t) $$
for some smooth functions $\Psi_1,\Psi_2$.
Therefore, by re-defining $t$ properly, we may assume either
\begin{subequations}\label{CPDCase1Spacelikemall}
\begin{equation}\label{CPDCase1Spacelikem1} 
m(s,t)=\int^{s}_{s_0}{\sin \theta(\tau)d\tau}+\Psi(t), \Psi\in C^\infty(M), 
\end{equation}
or
\begin{equation}
\label{CPDCase1Spacelikem2} 
m(s,t)=m(t).
\end{equation}
\end{subequations}

\textbf{Case 1.} Let $m$ satisfies \eqref{CPDCase1Spacelikem1}. In this case, by considering the equation \eqref{CPDCASEISpacelikeLeviCivitaEq1ALL} with \eqref{CPDAftClm1Eq1SCase1}, we get the Levi-Civita connection of $M$ satisfies 
\begin{eqnarray} \nonumber
\nabla _{\partial_s }\partial_s = 0, \quad \nabla _{\partial_s }\partial_t = \nabla _{\partial_t }\partial_s =\frac{m_s}{m}\partial_t, 
\quad \nabla _{\partial_t }\partial_t = -m{m_s}\partial_s+\frac{m_t}{m}\partial_t.
\end{eqnarray}%
By combining the first equation given above with \eqref{CPDCASEISpacelikeShapeOpm1} and using Gauss formula \eqref{MEtomGauss}, we have
\begin{eqnarray}
\label{CPDCase1Spacelikexss} x_{ss}&=&-\theta' e_3.
\end{eqnarray}
On the other hand, we have $\left\langle x_s,k\right\rangle=\cos \theta$ and $\left\langle x_t,k\right\rangle=0$ from the decomposition \eqref{CPDExpofS100}. By considering these equations, we see that  $x$ has the form of
\begin{equation}\label{CPDCase1Spacelikex}
x(s,t)=\left(\int^{s}_{s_0}{\cos \theta(\tau)d\tau},x_2(s,t),x_3(s,t),x_4(s,t)\right)+\gamma(t)
\end{equation}
for a $\mathbb E^4$-valued smooth function $\gamma=\left(0,\gamma_{2},\gamma_{3},\gamma_{4}\right)$.
 On the other hand, by considering \eqref{CPDAftClm1Eq1SCase1} and \eqref{CPDCase1Spacelikexss} in \eqref{CPDExpofS100}, we yield 
\begin{equation}\label{CPDCase1Spacelikex2} 
(1,0,0,0)=\cos \theta x_s-\frac{\sin \theta}{\theta'}x_{ss}.
\end{equation}

By solving \eqref{CPDCase1Spacelikex2} and considering $\left\langle x_s,x_s\right\rangle=1$, we obtain
\begin{align}\label{CPDCase1Spacelikenewx}
\begin{split}
x(s,t)=&\int^{s}_{s_0}{\cos \theta(\tau)d\tau}\Big(1,0,0,0\Big)+ \phi(t)\int^{s}_{s_0}{\sin \theta(\tau)d\tau}+\gamma(t),
\end{split}
\end{align}
where $\phi(t)=\Big(0,\phi_{2}(t),\phi_{3}(t),\phi_{4}(t)\Big)$ is the curve lying on $\mathbb S^3(1)$ in $\mathbb E^4$. 
Now, by considering $\displaystyle x_{st}=\frac{m_s}{m}x_t$ in \eqref{CPDCase1Spacelikenewx}, we can rewrite this parametrization as 

\begin{align}\label{CPDCase1Spacelikenewx2}
\begin{split}
x(s,t)=&\int^{s}_{s_0}{\cos \theta(\tau)d\tau}\Big(1,0,0,0\Big)+\phi(t) \int^{s}_{s_0}{\sin \theta(\tau)d\tau}+\int^t_{t_0}{\Psi(\tau)\phi^{\prime}(\tau)d\tau},
\end{split}
\end{align}
where $\Psi=\Psi(t)$ is a smooth function and $'$ denotes ordinary differentiation with respect to the parameter $t$. Also, since $\left\langle x_t,x_t\right\rangle=m^2$, we yield the curve $\phi$ parameterized by arc-lenght parameter $t$. Thus, we have the Case (1) of the theorem.

\textbf{Case 2.}  Let $m$ satisfy \eqref{CPDCase1Spacelikem2}. Here, we can take $m(t)=1$ by re-defining $t$ properly. In this case, the induced metric given in \eqref{CPDClassThmDiagonSpacelikekDefgEqRESCase1} of $M$ becomes $g=ds^2+dt^2$, the Levi Civita connection of $M$ satisfies 
\begin{equation}
\nabla _{\partial_s }\partial_s =0,\quad \nabla _{\partial_s }\partial_t=0, \quad%
\nabla _{\partial_t}\partial_t=0.
\end{equation}%
Also, considering $m=1$ in \eqref{CPDClassThmDiagonSpacelikekCod1Case11b} and \eqref{CPDCase1Spacelikems}, thus \eqref{CPDCASEISpacelikeLeviCivitaEq1} becomes
\begin{align} \label{CPDCASEISpacelikeShapeOpm2}
\begin{split}
S_3=\left(\begin{array}{cc}
-\theta'&0\\
0&0
\end{array}\right), \quad
&S_4=\left(\begin{array}{cc}
0&0\\
0&1
\end{array}\right).
\end{split}
\end{align}
Therefore, $x$ and the normal vectors $e_3,e_4$ satisfy
\begin{eqnarray}\nonumber
\nonumber x_{ss}=-\theta' e_3,\quad x_{st}=0,&& \quad x_{tt}= e_4.\\\nonumber
(e_3)_s=-\theta'x_s,\quad (e_3)_t=0,\\\nonumber
(e_4)_s=0, \quad (e_4)_t=- x_t.
\end{eqnarray}
A straightforward computation yields that $M$ is congruent to the surface given in Case (2) of the theorem. Hence, the proof for the necessary condition is obtained.

The proof of sufficient condition follows from a direct computation.
\end{proof}


Now, assume that the function $\theta$ defined in  \eqref{CPDExpofS100} satisfied $e_1(\theta)=0$. In this case, Lemma \ref{CPDCase1ClassThmDiagonSpacelikekClm1} gives
\begin{lem}\label{Case1ClassThmDiagonSpacelikekClm10}
The Levi-Civita connection $\nabla$ of $M$ is given by
\begin{subequations} \label{CASEISpacelikeLeviCivitaEq1ALL0}
\begin{eqnarray}
\label{CASEISpacelikeLeviCivitaEq1a0}\nabla _{e_{1}}e_{1}=\nabla _{e_{1}}e_{2}=0, &&
\\\label{CASEILeviCivitaEq1b0}
\nabla _{e_{2}}e_{1}=\tan \theta h_{22}^3e_2, &\quad &\nabla _{e_{2}}e_{2}=-\tan \theta h_{22}^3e_1. 
\end{eqnarray}
\end{subequations}
and the matrix representations of shape operator $S$ of $M$ with respect to $\{e_1,e_2\}$ is
\begin{align}\label{CASEISpacelikeLeviCivitaEq10}
\begin{split}
S_3=\left(\begin{array}{cc}
0&0\\
0&h_{22}^3
\end{array}\right), \quad
&S_4=\left(\begin{array}{cc}
0&0\\
0&h_{22}^4
\end{array}\right)
\end{split}
\end{align}
and coefficients of the second fundamental form satisfying
\begin{subequations} \label{ClassThmDiagonSpacelikekCod1Case1ALL0}
\begin{eqnarray} \label{ClassThmDiagonSpacelikekCod1Case11a0}
e_1(h_{22}^3)=-\tan \theta (h_{22}^3)^2, &&
\\\label{ClassThmDiagonSpacelikekCod1Case11b0}
e_1(h_{22}^4)=-\tan \theta h_{22}^3 h_{22}^4, &&
\\\label{ClassThmDiagonSpacelikekCod1Case11c0}
h_{11}^3=0, \quad h_{11}^4=0, \quad h_{12}^3=0, \quad h_{12}^4=0.
\end{eqnarray}
\end{subequations}
Note that here the angle $\theta$ is a non-zero constant.
\end{lem}

Next, we obtain the following local coordinate system on a neighborhood of  a point $p\in M$.
\begin{lem}\label{Case1ClassThmDiagonSpacelikekClm120}
There exists a local coordinate system  $(s,t)$ defined in a neighborhood $\mathcal N_p$ of $p$ such that the induced metric of $M$ is
\begin{equation}\label{ClassThmDiagonSpacelikekDefgEqRESCase10}
g=ds^2+m^2dt^2
\end{equation}
for a smooth function $m$ satisfying
\begin{equation}\label{Case1ClassThmDiagonSpacelikekDefm0}
e_1(m)-m \tan \theta h_{22}^3=0.
\end{equation}
Here, the angle $\theta$ is a non-zero constant. Furthermore, the vector fields $e_1,e_2$ described above become $e_1=\partial_s$,  $\displaystyle e_2=\frac 1{m}\partial_t$ in $\mathcal N_p$.
\end{lem}

\begin{proof}
We have $[e_1,e_2]=-\tan \theta h_{22}^3e_2$ because of \eqref{CASEISpacelikeLeviCivitaEq1ALL0}. Thus, if $m$ is a non-vanishing smooth function on $M$ satisfying  \eqref{Case1ClassThmDiagonSpacelikekDefm0}, then we have $\displaystyle \left[e_1,me_2\right]=0$. Therefore, there exists a local coordinate system $(s,t)$ such that $e_1=\partial_s$ and  $\displaystyle e_2=\frac 1{m}\partial_t$. Thus, the induced metric of  $M$ is as given in \eqref{ClassThmDiagonSpacelikekDefgEqRESCase10}.
\end{proof}

Now, we are ready to obtain the classification theorem.
\begin{thm}\label{ClassThmDiagonSpacelikek0}
Let $M$ be a regular surface in $\mathbb E^4$. Let  $M$ be a surface endowed with a canonical principal direction relative to $k=(1,0,0,0)$ and assume that the function $\theta$ defined in \eqref{CPDExpofS100} is   constant. Then, $M$ is congruent to the surface given by one of the followings
\begin{enumerate}
\item A surface given by
\begin{subequations}\label{ClassThmDiagonSpacelikekSurfaALL0}
\begin{equation}
\label{ClassThmDiagonSpacelikekSurfaEq10}
x(s,t)={s}\Big(\cos \theta,\phi_j(t)\sin \theta\Big)+\gamma(t), \quad j=2,3,4
\end{equation}
where $\gamma$ is the $\mathbb E^4$-valued function given by 
\begin{equation}\label{ClassThmDiagonSpacelikeSurfaEq20}
\gamma(t)=\Big(0,\sin \theta \int^t_{t_0}{\phi^{\prime}(\tau)\Psi (\tau)d\tau}\Big).
\end{equation}
Here, $ \Psi\in C^{\infty}(M)$ and $\phi$ is the unit speed curve lying on $\mathbb S^3(1)$ in $\mathbb E^4$ such that $\left\langle \gamma^{ \prime}(t),\phi(t) \right\rangle=0$ ;
\end{subequations}

\item A flat surface given by
\begin{align}\label{ClassThmDiagonSpacelikekSurfbEq10}
\begin{split}
x(s,t)={s} \Big({\cos \theta},\phi_j(t_0){\sin \theta}\Big)+\phi(t), \quad j=2,3,4
\end{split}
\end{align}
where $\phi(t_0)=(0,\phi_j(t_0))$ lying on $S^3(1)$ in $\mathbb E^4$ is a constant vector perpendicular to the vector $(1,0,0,0)$.
\end{enumerate}

Conversely, surfaces described above are CPD relative to $k=(1,0,0,0).$

\end{thm}

\begin{proof}
 Let $\{e_1,e_2;e_3,e_4\}$ be the local orthonormal frame field and coefficients of the second fundamental form described as before in Lemma \ref{Case1ClassThmDiagonSpacelikekClm10}, $(s,t)$ a local coordinate system given in Lemma \ref{Case1ClassThmDiagonSpacelikekClm120}.

Note that \eqref{ClassThmDiagonSpacelikekCod1Case11a0}, \eqref{ClassThmDiagonSpacelikekCod1Case11b0} and \eqref{Case1ClassThmDiagonSpacelikekDefm0} become, respectively
\begin{eqnarray}
\label{Case1SpacelikeCods0}(h_{22}^3)_s=-\tan \theta (h_{22}^3)^2,\\
\label{Case1SpacelikeCods20} (h_{22}^4)_s+\tan \theta h_{22}^3 h_{22}^4 =0,\\
\label{Case1Spacelikems0} m_s-m\tan \theta h_{22}^3=0.
\end{eqnarray}
Moreover, we have 
\begin{equation}\label{AftClm1Eq1SCase10}
e_1= x_s.
\end{equation}
By combining \eqref{Case1Spacelikems0} with \eqref{CASEISpacelikeLeviCivitaEq10} we obtain
the shape operator $S$ of $M$  as
\begin{align} \label{CASEISpacelikeShapeOpm10}
\begin{split}
S_3=\left( 
\begin{array}{cc}
0 & 0 \\ 
0 & \cot \theta\frac{m_s}{m}%
\end{array}%
\right)\quad
&S_4=\left(\begin{array}{cc}
0&0\\
0&\frac{1}{m}
\end{array}\right)
\end{split}
\end{align}
where $\theta$ is a non-zero constant.

By combining \eqref{Case1Spacelikems0} and \eqref{Case1SpacelikeCods0} we get
$$m(s,t)=\Psi_1(t)\Big({s}+\Psi_2(t)\Big) $$
for some smooth functions $\Psi_1,\Psi_2$.
Therefore, by re-defining $t$ properly, we may assume either
\begin{subequations}\label{Case1Spacelikemall0}
\begin{equation}\label{Case1Spacelikem10} 
m(s,t)=\sin \theta({s}+\Psi(t)), \Psi\in C^\infty(M), 
\end{equation}
or
\begin{equation}
\label{Case1Spacelikem20} 
m(s,t)=m(t).
\end{equation}
\end{subequations}

\textbf{Case 1.} Let $m$ satisfies \eqref{Case1Spacelikem10}. In this case, by considering the equation \eqref{CASEISpacelikeLeviCivitaEq1ALL0} with \eqref{AftClm1Eq1SCase10}, we get the Levi-Civita connection of $M$ satisfies 
\begin{eqnarray} \nonumber
\nabla _{\partial_s }\partial_s = 0, \quad \nabla _{\partial_s }\partial_t = \nabla _{\partial_t }\partial_s =\frac{m_s}{m}\partial_t, 
\quad \nabla _{\partial_t }\partial_t = -m{m_s}\partial_s+\frac{m_t}{m}\partial_t.
\end{eqnarray}%
 By combining these equations with \eqref{CASEISpacelikeShapeOpm10} and using Gauss formula \eqref{MEtomGauss}, we obtain
\begin{eqnarray}
\label{Case1Spacelikexss0} x_{ss}&=&0.
\end{eqnarray}
On the other hand, from the decomposition \eqref{CPDExpofS100}, we have $\left\langle x_s,k\right\rangle=\cos \theta$ and  $\left\langle x_t,k\right\rangle=0.$ By considering these equations, we see that  $x$ has the form of
\begin{equation}\label{Case1Spacelikex0}
x(s,t)=\Big(s {\cos \theta},x_j(s,t)+\gamma_j(t)\Big), \quad j=2,3,4.
\end{equation}
Here $\gamma(t)=\left(0,\gamma_j(t)\right)$ is a $\mathbb E^4$-valued smooth function. On the other hand, since \eqref{Case1Spacelikexss0} and $\left\langle x_s,x_s\right\rangle=1$, we get $\phi(t)$ is a curve lying on $S^3(1)$ in $\mathbb E^4$ with $\phi(t)=(0,\phi_j(t))$. So, if the parametrization reorder, we get
\begin{align}\label{Case1Spacelikenewx0}
\begin{split}
x(s,t)=&s \Big(\cos \theta, \phi_j(t){\sin \theta}\Big)+\gamma(t).
\end{split}
\end{align}
Now, by considering $\displaystyle x_{st}=\frac{m_s}{m}x_t$ in \eqref{Case1Spacelikenewx0}, we can rewrite the parametrization as 

\begin{align}\label{Case1Spacelikenewx20}
\begin{split}
x(s,t)=&{s}\Big(\cos \theta,\phi_j(t){\sin \theta}\Big)+\sin \theta \int_{t_0}^t{\Psi(\tau)\phi^{\prime}(\tau)d\tau},
\end{split}
\end{align}
where $\Psi=\Psi(t)$ is a smooth function. Also, since $\left\langle x_t,x_t\right\rangle=m^2$, we yield the curve $\phi$ parameterized by arc-lenght parameter $t$. Thus, we have the Case (1) of the theorem.

\textbf{Case 2.}  $m$ is given as \eqref{Case1Spacelikem20}. In this case, the induced metric of $M$ becomes $g=ds^2+dt^2$, the Levi Civita connection of $M$ satisfies 
\begin{equation}
\nabla _{\partial_s }\partial_s =0,\quad \nabla _{\partial_s }\partial_t=0, \quad%
\nabla _{\partial_t}\partial_t=0.
\end{equation}%
Also, considering \eqref{Case1Spacelikem20} in \eqref{ClassThmDiagonSpacelikekCod1Case11b0} and \eqref{Case1Spacelikems0}, thus \eqref{CASEISpacelikeLeviCivitaEq10} becomes
\begin{align} \label{CASEISpacelikeShapeOpm20}
\begin{split}
S_3=\left(\begin{array}{cc}
0&0\\
0&0
\end{array}\right), \quad
&S_4=\left(\begin{array}{cc}
0&0\\
0&1
\end{array}\right)
\end{split}
\end{align}
where $K(t)$ is a smooth function. Therefore, $x$ and the normal vectors $e_3,e_4$ satisfy
\begin{eqnarray}\nonumber
\nonumber x_{ss}=0,\quad x_{st}=0, \quad x_{tt}= e_4.\\\nonumber
(e_3)_s=0,\quad (e_3)_t=0,\\\nonumber
(e_4)_s=0, \quad (e_4)_t=- x_t.
\end{eqnarray}
A straightforward computation yields that $M$ is congruent to the surface given in Case (2) of the theorem. Hence, the proof for the necessary condition is obtained.

The proof of sufficient condition follows from a direct computation.
\end{proof}

\section*{Acknowledgments}
This paper is a part of PhD thesis of the first named author who is supported by The Scientific and Technological Research Council of Turkey (TUBITAK) as a PhD scholar.


\end{document}